\documentclass[12pt,oneside,reqno]{amsart}

      \usepackage{amssymb,enumerate,bbm,url}
\usepackage{tikz}
\usepackage{subfig}

      \theoremstyle{plain}
      \newtheorem{theorem}{Theorem}[section]
      \newtheorem{lemma}[theorem]{Lemma}
      
      \newtheorem{assumption}[theorem]{Assumption} 		
	  \newtheorem{proposition}[theorem]{Proposition}
     
      \theoremstyle{definition}
      \newtheorem{definition}[theorem]{Definition}

      \theoremstyle{remark}
      \newtheorem{remark}[theorem]{Remark}
	  \newtheorem{example}[theorem]{Example}

      \makeatletter
      \def\@setcopyright{}

      \def\serieslogo@{}
      \makeatother

\usepackage{geometry}
\usepackage{marginnote}
\usepackage{verbatim}

\usepackage{color}

	\usepackage{mathtools}	
\usepackage{bbm}

\newcommand{\tr}{\operatorname{tr}}
\newcommand{\N}{{\mathbb N}}
\newcommand{\Z}{{\mathbb Z}}
\newcommand{\R}{{\mathbb R}}

\renewcommand{\P}{{\mathbb{P}}}
\newcommand{\E}{{\mathbb{E}}}
\newcommand{\cov}{\mathrm{Cov}}
\begin{document}

 \author{Matthias L\"owe}
   \address{Inst. for Math. Stat., Univ. M\"unster, Orl\'{e}ans Ring 10, 48149 M\"unster, Germany}
   \email{maloewe@uni-muenster.de}

   \author{Kristina Schubert}
   \address{Inst. for Math. Stat., Univ. M\"unster, Orl\'{e}ans Ring 10, 48149 M\"unster, Germany}
   \email{kristina.schubert@uni-muenster.de}

\title[Spectrum of Random Matrices filled with Stochastic Processes]{On the Limiting Spectral Density of Random Matrices filled with Stochastic Processes}
%
%
%
   \begin{abstract}
We discuss the limiting spectral density of real symmetric random matrices. Other than in standard random matrix theory the upper diagonal entries are not assumed to be independent, but we will fill them with the entries of a stochastic process. Under assumptions on this process, which are satisfied, e.g., by stationary Markov chains on finite sets, by stationary Gibbs measures on finite state spaces, or by Gaussian Markov processes, we show that the limiting spectral distribution depends on the way the matrix is filled with the stochastic process.
If the filling is in a certain way compatible with the symmetry condition on the matrix, the limiting law of the empirical eigenvalue distribution is the well known semi-circle law. For other fillings we show that the semi-circle law cannot be the limiting spectral density.
\end{abstract}

   \maketitle

\setcounter{equation}{0}

\section{Introduction}
A central aspect in the study of random matrices with growing dimension is the analysis of their  eigenvalue distribution.
To introduce them let, for any $N\in\N$,  $\left\{X_N(i,j), 1\leq i\leq j\leq N\right\}$ be a real valued random field.
Define the symmetric random $N\times N$ matrix $A_N$ by
\begin{equation*}
A_N (i,j) = \frac{1}{\sqrt{N}} X_N(i,j), \qquad 1\leq i\leq j\leq N.
\end{equation*}
We will denote the (real) eigenvalues of $ A_{N}$ by $\lambda_1^{(N)} \le \lambda_2^{(N)} \le \ldots \le \lambda_N^{(N)}$. Let $\mu_N$ be the empirical eigenvalue distribution, i.e.
\begin{equation*}
\mu_N = \frac{1}{N} \sum_{k=1}^N \delta_{\lambda_k^{(N)}}.
\end{equation*}
In  \cite{Wigner_ursprung} and \cite{Wigner} Wigner proved that, if $X_N(p,q), 1 \leq p \leq q \leq N,$ are independent random variables with expectation 0, and either Bernoulli or normally distributed with variance 1 for off-diagonal elements and variance 2 on the diagonal, the empirical eigenvalue distribution $\mu_N$ converges weakly in probability to the so called semicircle distribution (or law), i.e.~the probability distribution $\nu$ on $\mathbb{R}$ with density
\begin{equation}\label{semi-circ}
\nu(dx)= \frac 1{2\pi} \sqrt{4-x^2} \mathbf{1}_{|x|\le 2}.
\end{equation}

Arnold \cite{Arnold} generalized this result by showing that
the convergence to the semicircle law is also  true, if one replaces the Bernoulli or Gaussian distribution by the assumption that the random variables are independent and identically distributed (i.i.d.) and have a finite fourth moment. This observation can be considered the starting point of a collection of results showing the universality of the semicircle law. For example, also the identical distribution of the random variables may be replaced by a weaker assumptions (see e.g.~\cite{erdoes_survey}). Moreover, Erd{\H{o}}s et al.~{\cite{ESY}) and Tao and Vu \cite{taovu2} observed  that the convergence of the empirical eigenvalue distribution towards the semi-circle law holds under general assumptions in a local sense.
The interested reader is referred to \cite{taovu} for more results on universality of the semi-circle law, in particular the so called ``four moment theorem''.

Another generalizations of Wigners semi-circle law deals with matrix ensembles with entries realized according to weighted Haar measures on classical
(e.g.~orthogonal, unitary, symplectic) groups. Such results lay a bridge between classical and non-commutative probability
(see e.g.~\cite{alice_stflour}, or the recommendable book by Anderson, Guionnet, and Zeitouni \cite{agz}).

A different approach to universality was taken in \cite{Schenker_Schulz-Baldes}, \cite{gotze_tikho} and \cite{Diag}. In all these articles random matrices with correlated entries are studied. While the first of these papers imposes the condition that the number of correlated matrix entries satisfies a certain bound, the second assumes that the entries of the random matrix satisfy a certain martingale condition. Moreover, \cite{gotze_tikho} establishes a Stein approach to the semicircle law. In \cite{Diag} it  is shown that, if the diagonals of $X_N$ are independent and the correlation between elements along a diagonal decays sufficiently quickly, again the limiting spectral distribution is the semi-circle law.
This result has to be compared with the situation in \cite{FL_CW} and \cite{FL_CW2}, where the diagonals are still independent, but the random variables along a diagonal are exchangeable. Here again, for weak correlations one finds the semicircle as the limiting spectral density for the eigenvalues. However, for stronger correlations the limiting spectral measure is a free convolution of the semicircle law with the limit measure for the eigenvalues of a normalized Toeplitz matrix, which was found in \cite{brycdembo}. This means that one observes a kind of phase transition in the limiting spectral density. A similar results for matrices of Hankel type was shown in \cite{Schubert_Hankel}. 
In \cite{Kirsch_et_al} the authors generalize the conditions from \cite{FL_CW} by assuming that now the entire upper diagonal entries of the matrix are exchangeable. Again, if the
correlations decay quickly enough, the limiting spectral distribution is the semicircle law. This is, to the best of our knowledge, one of the rare occasions that  the semi-circle law has been shown for a random matrix without any independence assumptions.

The results in \cite{Diag} and \cite{Kirsch_et_al} immediately raise the following question: What is the limiting spectral distribution of a random matrix, if all the upper diagonal entries are correlated but not exchangeable. Considering the matrix entries as realizations of a stochastic process,  a natural assumption is again that the correlations decay quickly, e.g.~like the correlations in a Markov chain. This is exactly the problem we are going to attack in the present paper. In a nutshell, the result is that under natural assumptions, e.g.~centered entries if the stochastic process is a Markov chain, the semicircle law will be the limiting spectral distribution. However, this results depends on the way the matrix is filled with the realization of the stochastic process. We will need a condition that is e.g.~satisfied, if the stochastic process follows the diagonals of the random matrix. We will also show that there are Markov chains (even on $\{-1, +1\}$) that, if filled into a matrix row- or column-wise, lead to a spectral distribution that does not converge to the semicircle law.

We organize this paper in the following way: The second section contains our basic definitions together with the central results. In Section 3 we focus on the assumptions for our results. We show that suitably chosen Markov chains,  Gibbs measures and Gaussian Markov processes are examples for the stochastic processes that appear in our main theorem and we consider the filling, where the stochastic process follows the diagonals of the matrix. Section 4 contains the proofs of our results.

\section{Random Matrices with Entries from a Stochastic Process}
Given a stochastic process $(Z_n)_n$, there are, of course, various ways to fill them into a (symmetric) random matrix. The following definition formalizes these ways. It also introduces a distance on the indices of the matrix, which will enable us to formulate a condition, under which the spectral distribution of the corresponding matrices converges to the semicircle law. This distance is induced by the filling of the matrix.

\begin{definition}\label{Def_filling}\text{}
\begin{enumerate}
\item[(a)]
A  symmetric matrix $X_N \in \mathbb R^{N \times N}$ is generated by a stochastic process $(Z_n)_n$ and a mapping $\varphi_N$, if
\begin{enumerate}
\item[(i)] $\varphi_N: \{ 1,\ldots, N(N+1)/2\} \to \{(i,j) \in  \{1,\ldots, N\}^2 : i \leq j\}$ is bijective,
\item[(ii)] the matrix entries (in the upper triangular part of the matrix) are given by  $(Z_n)_n$ via
$$X_N(i,j)=Z_{\varphi_N^{(-1) }(i,j)}, \quad 1\leq i\leq j\leq N.$$
\end{enumerate}
\item[(b)] We measure the distance between two matrix indices $(i,j)$ and $(i',j')$ of a $N\times N$ matrix in terms of their distance in the stochastic process, i.e.~we set
$$ \|(i,j)- (i',j') \|_{\varphi_N} \coloneqq |\varphi_N^{(-1) }(i,j)-\varphi_N^{(-1) }(i',j') |, \quad i\leq j, i' \leq j'.$$
If $i>j$ or $i'>j'$, we use $(j,i)$ resp.~$(j',i')$ instead to calculate the distance.
\item[(c)] We call matrix indices $(i,j), (i',j') \in \{1,\ldots N\}^2$, $i \leq j$, $i' \leq j'$ (horizontal or vertical, respectively) neighbors if
\begin{itemize}
\item $i=i'$ and $j\in \{j'+1, j'-1\}$  or
\item  $j=j'$ and $i\in \{i'+1, i'-1\}$, respectively.
\end{itemize}
Considering the path through the upper triangular part of a matrix prescribed by the mapping $\varphi_N$, we denote the number of  steps to the horizontal or vertical neighboring matrix entry by
$$J(\varphi_N) \coloneqq \# \left \{1\leq i < \frac{ N(N+1)}{2}: \varphi_N(i) \text{ and }\varphi_N(i+1) \text{ are neighbors} \right \}.$$
\end{enumerate}
\end{definition}
In our main theorem we study the convergence of the empirical eigenvalue distribution of matrices generated by a stochastic process $(Z_n)_n$ with correlations that decay quickly  and a mapping $\varphi_N$.
We provide conditions on the mapping $\varphi_N$ and the stochastic process $(Z_n)_n$ under which $\mu_N$ converges weakly in probability 
to Wigner's semicircle law $\nu$.
These condition are:
\begin{assumption}\label{schranke_eintr_1}
Assume that the mapping $\varphi_N$ satisfies
\begin{equation*}
\forall i,j, n :\# \{ x \in \{1,\ldots N\} : \|(i,x)- (x,j) \|_{\varphi_N} =n\} =o(N).
\end{equation*}
\end{assumption}
\begin{assumption}\label{bedingungen_proc}
Let $(Z_n)_{n \in \N} $ be a stochastic process.
Assume the following:
\begin{enumerate}
\item[(i)]
for all $i \in \mathbb N$ we have $\mathbb E(Z_i^2)=1$,
\item[(ii)]
for all odd numbers $k$ and all $i_1, \ldots, i_k$ we have that
\begin{equation}
\label{moment_cond1}
\mathbb E (Z_{i_1} Z_{i_2} \ldots Z_{i_k})=0,
\end{equation}
\item[(iii)] for all even $k$ there are constants $C>0$ and $\beta \in[0,1)$ such that for all $i_1 \le i_2 \le \ldots \le i_k$ with  $n_j \coloneqq i_{j+1}-i_j$
 we have
\begin{equation}
\label{Beh_EWert}
\left|\mathbb E (Z_{i_1} Z_{i_2} \ldots Z_{i_k})\right|   \le C \beta^{n_1+n_3+\ldots + n_{k-1}},
\end{equation}
for $j_1 \le \ldots \le j_k$ with $d \coloneqq \min_{n,m \in \{1,\ldots,k\}}|i_n-j_m|$  we have
 \begin{equation}
\label{est_corr}
\left| \mathbb E (Z_{i_1} \ldots Z_{i_{k}} Z_{j_1} \ldots Z_{j_{k}}) - \mathbb E (Z_{i_1} \ldots Z_{i_{k}}) \mathbb E (Z_{j_1} \ldots Z_{j_{k}})   \right| \leq C  \beta^{d},
\end{equation}
and for  $d'\coloneqq \min_{j=1,3,5\ldots} i_{j+2}-i_j$ we have
\begin{equation}
\label{est_squares_2}
\left| \mathbb E (Z_{i_1}^2 Z_{i_3}^2 \ldots Z_{i_{k-1}}^2)-1 \right| \leq C  \beta^{d'}.
\end{equation}
\end{enumerate}
\end{assumption}

\begin{remark}
If, instead of \eqref{est_corr} we assume that for all $k$ there exists $C>0$ and $\beta \in [0,1)$ such that for all $i_1 \le i_2 \le \ldots \le i_k$ and $j_1 \le \ldots \le j_l$ with $l \leq k$  and $d \coloneqq \min_{\substack{ n\in \{1,\ldots,k\} \\ m \in \{1, \ldots, l\}}}|i_n-j_m|$  we have
 \begin{equation*}
\left| \mathbb E (Z_{i_1} \ldots Z_{i_{k}} Z_{j_1} \ldots Z_{j_{l}}) - \mathbb E (Z_{i_1} \ldots Z_{i_{k}}) \mathbb E (Z_{j_1} \ldots Z_{j_{l}})   \right| \leq C  \beta^{d},
\end{equation*}
the estimate \eqref{est_squares_2} follows automatically. However, all these assumptions are met by our examples.
\end{remark}
With these assumptions our main theorem reads as follows.
\begin{theorem}\label{main_theo}
Let
$X_N \in \mathbb R^{N \times N}$ be symmetric matrices generated by a stochastic process $(Z_n)_n$ that fulfills Assumption \ref{bedingungen_proc} and a filling $\varphi_N$ that satisfies Assumption \ref{schranke_eintr_1}.
Consider the rescaled matrix $A_N(i,j)\coloneqq \frac{1}{\sqrt{N}} X_N (i,j)$.
Then, 
the empirical spectral distribution $\mu_N$ of $A_N$
converges weakly in probability to the semicircle law $\nu$ as given by (\ref{semi-circ}).
\end{theorem}

\begin{remark}
Conditions \eqref{Beh_EWert},  \eqref{est_corr}, and  \eqref{est_squares_2} of Assumption \ref{bedingungen_proc} ensure that the correlations of entries of the stochastic process decay sufficiently fast.
A prominent example, in which these conditions are satisfied, is the case where $(Z_n)_n$ is an ergodic Markov chain on a finite subset of $\R$ started in its stationary measure $\rho$. I.e.~$\P^{Z_1} = \rho$, and additionally we assume that \eqref{moment_cond1} holds true.
The verification of \eqref{Beh_EWert}, \eqref{est_corr}, and  \eqref{est_squares_2}  in this case is given in Lemma \ref{MClemma1} below.

However, there also other interesting processes that satisfy conditions \eqref{Beh_EWert}, \eqref{est_corr}, and \eqref{est_squares_2}, e.g.~certain one-dimensional Gibbs measures on finite state spaces. Moreover, also certain Gaussian processes fall into the realm of Theorem \ref{main_theo}. These examples will also be treated in the next section.

Regarding the mapping $\varphi_N$, we will see in Example \ref{example_diag}  and Lemma \ref{lemma_diag} that Assumption \ref{schranke_eintr_1} is non-empty. It is e.g.~satisfied if we fill the matrix with the entries of the stochastic process one diagonal after another (starting with the main diagonal and always proceeding from top to bottom), i.e.~$\varphi_N(1)=(1,1), \varphi_N(2)=(2,2), \ldots, \varphi_N(N)=(N,N),\varphi_N(N+1)=(1,2), \varphi_N(N+2)=(2,3), \ldots .$ In view of the results in \cite{Diag}, this is a natural filling. On the other hand, we will also see, that there are fillings for which the convergence from the previous theorem is not true.
\end{remark}
\begin{remark}
For a stochastic process that satisfies the conditions of Assumption~\ref{bedingungen_proc}, we can conclude a further estimate for the left hand side of \eqref{est_corr}:
Let $i_1 \leq \ldots \leq i_k$ and $j_1 \leq  \ldots \leq j_k$ and let $l_1 \leq \ldots \leq l_{2k}$ be given by
$(l_1,\ldots, l_{2k})\coloneqq \text{sort}(i_1,\ldots, i_k,j_1,\ldots, j_k)$, where $\text{sort}:\mathbb R^{2k} \to \mathbb R^{2k} $ denotes the function  that sorts the arguments in increasing order.
We claim that for $L \coloneqq (l_2-l_1)+(l_4-l_3)+\ldots + (l_{2k}-l_{2k-1})$ we have
\begin{align}
&\left| \mathbb E (Z_{i_1} \ldots Z_{i_{k}} Z_{j_1} \ldots Z_{j_{k}}) - \mathbb E (Z_{i_1} \ldots Z_{i_{k}}) \mathbb E (Z_{j_1} \ldots Z_{j_{k}})   \right| \nonumber
\\
\leq &
\left| \mathbb E (Z_{i_1} \ldots Z_{i_{k}} Z_{j_1} \ldots Z_{j_{k}})| + | \mathbb E (Z_{i_1} \ldots Z_{i_{k}})| \cdot | \mathbb E (Z_{j_1} \ldots Z_{j_{k}})   \right|
\label{est_joint_indices}
\\\leq & C \beta^{L}.\label{est_joint_indices_2}
\end{align}
The estimate for the first term in \eqref{est_joint_indices} is obvious from \eqref{Beh_EWert}. For the second term in \eqref{est_joint_indices} we use again \eqref{Beh_EWert} for each of the two factors together with
\begin{multline*}
(i_2-i_1)+(i_4-i_3)+ \ldots + (i_k-i_{k-1})\\
 \quad +(j_2-j_1)+(j_4-j_3)+ \ldots + (j_k-j_{k-1}) \geq L.
\end{multline*}
This can be seen by iterating the simple argument that for any $a\leq b$ and $a' \leq  b'$ we have
\begin{equation*}
(b-a) + (b'-a') \geq \begin{cases} (a'-a) + (b-b') & a \leq a' \leq b' \leq b \\ (a'-a)+(b'-b)  & a \leq a' \leq b \leq b' \end{cases}.
\end{equation*}
Combining \eqref{est_corr} and \eqref{est_joint_indices_2} then gives
\begin{align}
\left| \mathbb E (Z_{i_1} \ldots Z_{i_{k}} Z_{j_1} \ldots Z_{j_{k}}) - \mathbb E (Z_{i_1} \ldots Z_{i_{k}}) \mathbb E (Z_{j_1} \ldots Z_{j_{k}})   \right|
& \leq C \beta^{\max(d,L)} \nonumber
\\
& \leq C \beta^{\frac{1}{2} (d + L)}.
\label{est_corr_combined}
\end{align}
\end{remark}
The number of ways to prove a result like Theorem \ref{main_theo} is limited. As it is well known that
Wigner's semicircle law $\nu$ is uniquely determined by its moments, we will use the method of moments, which is one of the fundamental tools when proving limit theorems for dependent random variables. The central step in the proof is thus to show that the moments of the empirical spectral measure converge to the moments of $\nu$. These  moments are given by the Catalan numbers  $\kappa_k \coloneqq \frac{1}{k+1} \binom{2k}{k}$, i.e.~(see e.g.~\cite{agz})
\begin{equation*}
\int x^k d\nu(x) =\begin{cases} \kappa_{k/2},& k \text{ even} \\ 0,& \text{ otherwise.} \end{cases}
\end{equation*}

We note that Assumption \ref{schranke_eintr_1} of Theorem \ref{main_theo} is sufficient but not necessary. In the following proposition, we provide a condition on the mapping $\varphi_N$, under which the empirical eigenvalue distribution does not converge to the semi-circle law, if the underlying process is e.g.~a certain Markov chain with state space $\{-1,1\}$. This convergence is, of course, already impossible, if the fourth moment of the respective distribution does not converge to $\kappa_2$.

\begin{proposition}\label{non_conv}
Let $A_N$ be given as in Theorem \ref{main_theo}. Assume the following: 
\begin{enumerate}
\item[(i)]
There exists $c>0$ such that $J(\varphi_N) \geq c N^2$ for all $N \in \mathbb N$.
\item[(ii)]
There exist $\beta \in (0,1)$, and $C>0$ such that for all $i_1 \leq i_2 \leq i_3 \leq i_{4}$ and $n_j \coloneqq i_{j+1}-i_j$
\begin{equation}
\label{non_conv_condition1}
 \mathbb E (Z_{i_1} Z_{i_2}Z_{i_3}Z_{i_4} )= C \beta^{n_1+n_3}
\end{equation}
and
\begin{equation}
\label{non_conv_condition3}
 \mathbb E (Z_{i_1}^2 Z_{i_2}^2)=1.
\end{equation}
\end{enumerate}
Then we have  
$$\lim_{N \to\infty}\mathbb E \left( \frac{1}{N} \tr A_N^4 \right) \neq \kappa_2 .$$
Hence, the empirical spectral distribution $\mu_N$ of $A_N$ does not converge to $\nu$.
\end{proposition}

\begin{remark}\label{rem_filling}
The condition   $J(\varphi_N) \geq c N^2$  of Proposition \ref{non_conv} can be interpreted as follows: Considering the path in the upper triangular part of the matrix prescribed by $\varphi_N$, if there are `too many' steps from a matrix element to its horizontal or vertical neighbor, the limiting spectral density is not the semi-circle.

There are natural examples for fillings with this condition, the most prominent one is probably the mapping
\begin{eqnarray*}
&&\varphi_N(1)=(1,1), \varphi_N(2)=(1,2), \ldots, \varphi_N(N)=(1,N),\\ &&\qquad \varphi_N(N+1)=(2,2), \varphi_N(N+2)=(2,3), \ldots
\end{eqnarray*}
(i.e.~the stochastic process fills the upper triangular matrix row by row from left to right). For this particular mapping we have $J(\varphi_N)= \frac{N(N+1)}{2} - N+1=\frac{N(N-1)}{2}+1$.
\end{remark}

\section{Examples and a `diagonal' mapping $\varphi_N$}
In this section we show that the moment conditions in Assumption \ref{bedingungen_proc} are satisfied by certain examples, among them Markov chains on a finite state space, high temperature Gibbs measures in dimension one, and Gaussian processes. In Lemma \ref{lemma_expect}, we show that there are Markov chains on $\{1,-1\}$, that satisfy conditions \eqref{non_conv_condition1} and \eqref{non_conv_condition3}.
We further provide a mapping $\varphi_N$ that is valid in the sense of Assumption~\ref{schranke_eintr_1}. Hence, we show that none of our assumptions in Theorem \ref{main_theo} and Proposition \ref{non_conv} is empty. We start with the moment conditions in Assumption \ref{bedingungen_proc}.
\subsection{Markov chains}
\begin{lemma}\label{MClemma1}
Let $Z_n$ denote a stationary, ergodic Markov chain on a finite subset of $\R$ with invariant measure $\rho$. Assume that $\E (Z_i^2)=1$ for all $i \in \mathbb N$ and for all odd numbers $k$ and all $i_1, \ldots, i_k$ we have that
$$
\mathbb E (Z_{i_1} Z_{i_2} \ldots Z_{i_k})=0.
$$
Then  for all even $k$  there are constants $C>0$ and $\beta \in[0,1)$ such that for all $i_1 \le i_2 \le \ldots \le i_k$  it holds
\begin{equation}
\label{Beh_EWert2}
\left|\mathbb E (Z_{i_1} Z_{i_2} \ldots Z_{i_k})\right|   \le C \beta^{n_1+n_3+\ldots + n_{k-1}},\quad n_j \coloneqq i_{j+1}-i_j
\end{equation}
and
for $j_1 \le \ldots \le j_k$ with $d \coloneqq \min_{n,m \in \{1,\ldots,k\}}|i_n-j_m|$  we have
 \begin{equation}
\label{MC_correlations}
\left| \mathbb E (Z_{i_1} \ldots Z_{i_{k}} Z_{j_1} \ldots Z_{j_{k}}) - \mathbb E (Z_{i_1} \ldots Z_{i_{k}}) \mathbb E (Z_{j_1} \ldots Z_{j_{k}})   \right| \leq C  \beta^{d}.
\end{equation}
Moreover, for   $d'\coloneqq \min_{j=1,2,\ldots} i_{j+1}-i_j$
\begin{equation}
\label{second_moments}
\left| \mathbb E (Z_{i_1}^2 Z_{i_2}^2 \ldots Z_{i_{k}}^2)-1 \right| \leq C  \beta^{d'}.
\end{equation}  
\end{lemma}

\begin{proof}
Call the state space of the Markov chain  $S=\{s_1, \ldots, s_m\}$. Note that by standard arguments for any $l \le k$
$$
\max_{s_1, s_2 \in S } |\P(Z_{i_l}=s_2|Z_{i_{l-1}}=s_1)-\P(Z_{i_l}=s_2)|\le C \alpha^{i_l-i_{l-1}}
$$
for some constant $C>0$ and some $\alpha \in [0,1)$ (for a proof see \cite{Peres}, Theorem 4.9). By the Markov property thus also for any $l \le k$
\begin{eqnarray*}
&&\max_{s_{j_1}, \ldots, s_{j_k} \in S } |\P(Z_{i_k}=s_{j_k}, \ldots, Z_{i_l}=s_{j_l}|Z_{i_{l-1}}=s_{j_{l-1}}, \ldots, Z_{i_{1}}=s_{j_{1}} )\\
&&\qquad \qquad -\P(Z_{i_k}=s_{j_k}, \ldots, Z_{i_l}=s_{j_l})|\le C \alpha^{i_l-i_{l-1}}.
\end{eqnarray*}
Taking into account the finiteness of the state space $S$ and that
\begin{eqnarray*}
&& \E(Z_{i_k}^{a_{i_k}} \cdots  Z_{i_l}^{a_{i_l}}|Z_{i_{l-1}}\ldots  Z_{i_1})=\\
&&\qquad \sum_{s_{j_k}, \ldots s_{j_1} \in S} s_{j_k}^{a_{i_k}} \cdots s_{j_l}^{a_{i_l}} \mathbbm{1}_{\{Z_{i_{l-1}}=s_{j_{l-1}}, \ldots, Z_{i_{1}}=s_{j_{1}}\}} \times \\
&& \hskip 2cm \times
\P(Z_{i_k}=s_{j_k}, \ldots, Z_{i_l}=s_{j_l}|Z_{i_{l-1}}=s_{j_{l-1}}, \ldots, Z_{i_{1}}=s_{j_{1}} )
\end{eqnarray*}
for any $l \le k$ and any $i_1 \le \ldots \le i_{l-1} \le i_l < i_{l+1}< \ldots < i_k$ , we obtain
\begin{equation}\label{mixed moments conditioned}
\left| \E(Z_{i_k}^{a_{i_k}} \cdots  Z_{i_l}^{a_{i_l}}|Z_{i_{l-1}}\ldots  Z_{i_1})-\E(Z_{i_k}^{a_{i_k}} \cdots  Z_{i_l}^{a_{i_l}})\right| \le C \alpha^{i_l-i_{l-1}}
\end{equation}
for all integers $a_{i_1},\ldots a_{i_l}$.

Now, since all odd mixed moments vanish, we obtain for the $2k$'th moments with $i_1 \le i_2 \le \ldots \le i_{2k}$  and all $l\le k$
\begin{multline*}
\mathbb E (Z_{i_1} Z_{i_2} \ldots Z_{i_{2k}})
 =  \mathbb E (Z_{i_1} Z_{i_2} \ldots Z_{i_{2l-1}}
 \\
  [\mathbb E (Z_{i_{2l}} Z_{i_{2l+1}} \ldots Z_{i_{2k}}|Z_{i_1} Z_{i_2} \ldots Z_{i_{2l-1}})-
\mathbb E (Z_{i_{2l}}  \ldots Z_{i_{2k}})]).
\end{multline*}
Thus \eqref{mixed moments conditioned} guarantees that
$$
\left| \mathbb E (Z_{i_1} Z_{i_2} \ldots Z_{i_{2k}})\right| \le C \alpha^{i_{2l}-i_{2l-1}}
$$
for all $l \le k$.
Therefore with $\beta \coloneqq \alpha^{1/k}$
\begin{equation*}
\left| \mathbb E (Z_{i_1} Z_{i_2} \ldots Z_{i_{2k}})\right| \le C \alpha^{\max_{1 \le l \le k} \{i_{2l}-i_{2l-1}\}} \le C \prod_{l=1}^k \beta^{i_{2l}-i_{2l-1}}.
\end{equation*}
This completes the proof of \eqref{Beh_EWert2}.

We continue with the proof of  \eqref{second_moments}. We observe that \eqref{mixed moments conditioned} in particular gives
\begin{equation*}
\left| \E(Z_{i_k}^{2}|Z_{i_{k-1}}\ldots Z_{i_1})-\mathbb E(Z_{i_k}^2)\right| =\left| \E(Z_{i_k}^{2}|Z_{i_{k-1}}\ldots Z_{i_1})-1)\right| \le C \alpha^{i_k-i_{k-1}}.
\end{equation*}
By the same reasoning as above, together with the basic relation $\mathbb E(XY -1)= \mathbb E(X(Y-1))+ \mathbb E(X-1)$ (for any random variables $X$ and $Y$), we have
\begin{align*}
&|\mathbb E (Z_{i_1}^2 Z_{i_2}^2 \ldots Z_{i_{k}}^2)-1|
= |\mathbb E ( \, \mathbb E(Z_{i_1}^2 Z_{i_2}^2 \ldots Z_{i_{k}}^2| Z_{i_{k-1}}\ldots Z_{i_1}) -1) |
\\
& = |\mathbb E ( \, Z_{i_1}^2 Z_{i_2}^2 \ldots Z_{i_{k-1}}^2\mathbb E( Z_{i_{k}}^2|Z_{i_{k-1}}\ldots Z_{i_1}) -1) |
\\
& =  |\mathbb E ( \, Z_{i_1}^2 Z_{i_2}^2 \ldots Z_{i_{k-1}}^2 [\mathbb E( Z_{i_{k}}^2|Z_{i_{k-1}}\ldots Z_{i_1}) -1]) + \mathbb E ( \, Z_{i_1}^2 Z_{i_2}^2 \ldots Z_{i_{k-1}}^2 -1) |
\\
& \leq  C \alpha^{i_k-i_{k-1}} +  |\mathbb E ( \, Z_{i_1}^2 Z_{i_2}^2 \ldots Z_{i_{k-1}}^2 -1)|.
\end{align*}
Iterating this calculation shows \eqref{second_moments}.

We continue with the proof of \eqref{MC_correlations}. We denote by $l_1 \leq \ldots \leq l_{2k}$ the (sorted) joint indices consisting of $i_1,\ldots, i_k$ and $j_1,\ldots, j_k$, i.e.~$(l_1,\ldots, l_{2k})=\text{sort}(i_1,\ldots, i_k,j_1,\ldots, j_k)$, where $\text{sort}:\mathbb R^{2k} \to \mathbb R^{2k} $ denotes the function  that sorts the arguments in increasing order.
We have the general identity for arbitrary $1\leq n \leq 2k-1$
\begin{align}
\nonumber
&|\mathbb E(Z_{l_1} \ldots Z_{l_{2k}}) -  \mathbb E(Z_{l_1} \ldots Z_{l_{n}})E(Z_{l_{n+1}} \ldots Z_{l_{2k}})|
\\ \nonumber
 =& |\mathbb E(\mathbb E(Z_{l_1} \ldots Z_{l_{2k}}| Z_{l_1} \ldots Z_{l_n}) - Z_{l_1} \ldots Z_{l_n}\mathbb E(Z_{l_{n+1}} \ldots Z_{l_{2k}}) )|
\\ \nonumber
 = &|\mathbb E( Z_{l_1} \ldots Z_{l_n} [ \mathbb E(Z_{l_{n+1}} \ldots Z_{l_{2k}}| Z_{l_1} \ldots Z_{l_n}) - \mathbb E(Z_{l_{n+1}} \ldots Z_{l_{2k}})] )|
\\
 \leq & C \alpha^{l_{n+1}-l_n}, \label{factor_ecpectation}
\end{align}
where the last estimate is due to  (\ref{mixed moments conditioned}).
If $i_k \leq j_1$ or $j_k \leq i_1$ this proves the claim. To complete the proof for an arbitrary ordering of the indices, we apply the above estimate successively for certain values of $n$. Without loss of generality, we assume that $l_1=i_1$. Let $n_1,n_2,\ldots,n_m$ denote the indices such that
\begin{align*}
(l_1,\ldots, l_{n_1})&=(i_1,\ldots, i_{n_1})
\\
(l_{n_1+1},\ldots, l_{n_2})&=(j_1,\ldots, i_{n_2})
\\
(l_{n_2+1},\ldots, l_{n_3})&=(i_{n_1+1},\ldots, i_{n_1+(n_3-n_2)})
\\
(l_{n_3+1},\ldots, l_{n_4})&=(j_{n_2+1},\ldots, i_{n_2+(n_4-n_3)}), \ldots
\end{align*}
This means $(l_1,\ldots, l_{n_1}),(l_{n_1+1},\ldots, l_{n_2}), \ldots$ are the longest  subsequences of $(l_1,\ldots, l_{2k})$ such that each sub-sequence consists either of indices in $\{i_1,\ldots,i_k\}$ or in $\{j_1,\ldots,j_k\}$ only.
Hence, on the one hand, by iterating \eqref{factor_ecpectation}, we obtain for some (possibly different) constant $C$
\begin{align}
\nonumber
&|\mathbb E(Z_{l_1} \ldots Z_{l_{2k}}) - \mathbb E(Z_{l_1} \ldots Z_{l_{n_1}})\mathbb E(Z_{l_{n_1}+1} \ldots Z_{l_{n_2}}) \ldots  \mathbb E(Z_{l_{n_{m-1}}+1} \ldots Z_{l_{n_m}})|
\\
& \leq C \alpha^d. \label{corr_est1}
\end{align}
 On the other hand, we can apply the analogous factorization to both terms $\mathbb E (Z_{i_1} \ldots Z_{i_k})$ and   $\mathbb E (Z_{j_1} \ldots Z_{j_k})$ and use e.g.~$i_{n_1+1}-i_{n_1} \geq d, j_{n_2+1}-i_{n_2} \geq d, \ldots$ to obtain
 \begin{align}
& | \mathbb E (Z_{i_1} \ldots Z_{i_k})  \mathbb E (Z_{j_1} \ldots Z_{j_k})
\nonumber  \\
 & \quad - \mathbb E(Z_{l_1} \ldots Z_{l_{n_1}})\mathbb E(Z_{l_{n_1}+1} \ldots Z_{l_{n_2}}) \ldots  \mathbb E(Z_{l_{n_{m-1}}+1} \ldots Z_{l_{n_m}})|\leq C \alpha^d.
 \label{corr_est2}
 \end{align}
 Combining \eqref{corr_est1} and \eqref{corr_est2} completes the proof of \eqref{MC_correlations}.
\end{proof}
In order to see that the assumption of Proposition \ref{non_conv} is non-empty, we consider Markov chains with state space $\{-1,+1\}$  and $\mathbb P (Z_n=i|Z_{n-1}=i)=p$. 
As in this case trivially  $\E (Z_{i_1}^2  Z_{i_2}^2)=1$, it remains to verify \eqref{non_conv_condition1}. Lemma \ref{lemma_expect} shows  for $p> \frac{1}{2}$ that \eqref{non_conv_condition1} is valid with $C=1$ and  $\beta=2p-1$. 
Hence, such a Markov chain  together with a filling that satisfies  $J(\varphi_N) \geq c N^2$, e.g.~a row-wise filling, leads to a limiting spectral measure that differs from the semi-circle (see Fig.~\ref{different_fillings}). 
In addition, such a Markov chain also falls in the realm of Lemma \ref{MClemma1} and hence  together with a filling that satisfies Assumption \ref{schranke_eintr_1}, generates  random matrices, for which the limiting spectral measure is the semicircle. Thus, in this case, the filling $\varphi_N$ is essential for the limiting spectral measure. 
\begin{figure}
\centering
\hfill %
\subfloat[diagonal filling, see example \ref{example_diag} \label{pic:Bild1}]{\includegraphics[width=0.4\textwidth]{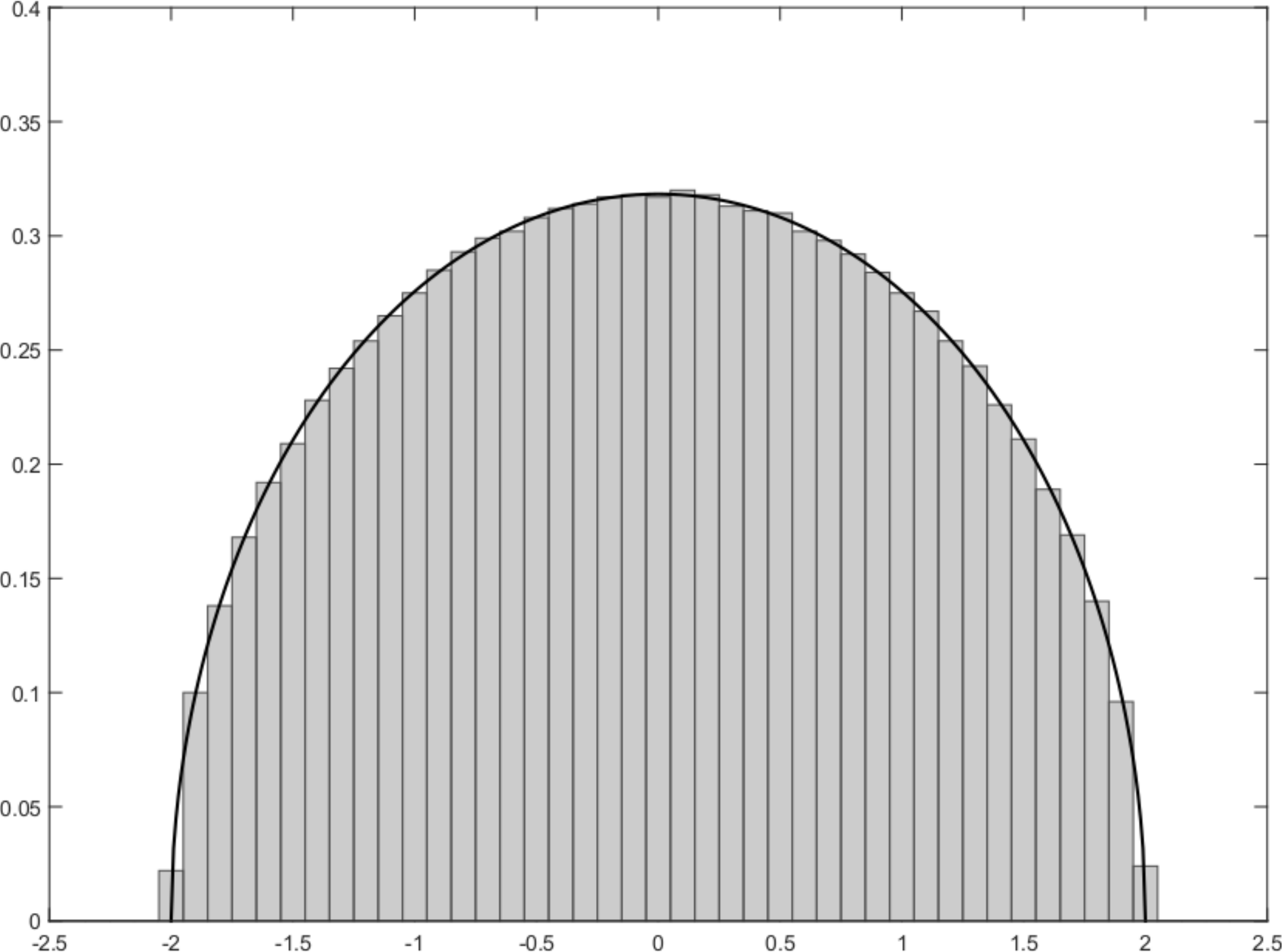}}
\hfill 
\subfloat[row-wise filling, see remark~\ref{rem_filling}\label{pic:Bild2}]{\includegraphics[width=0.4\textwidth]{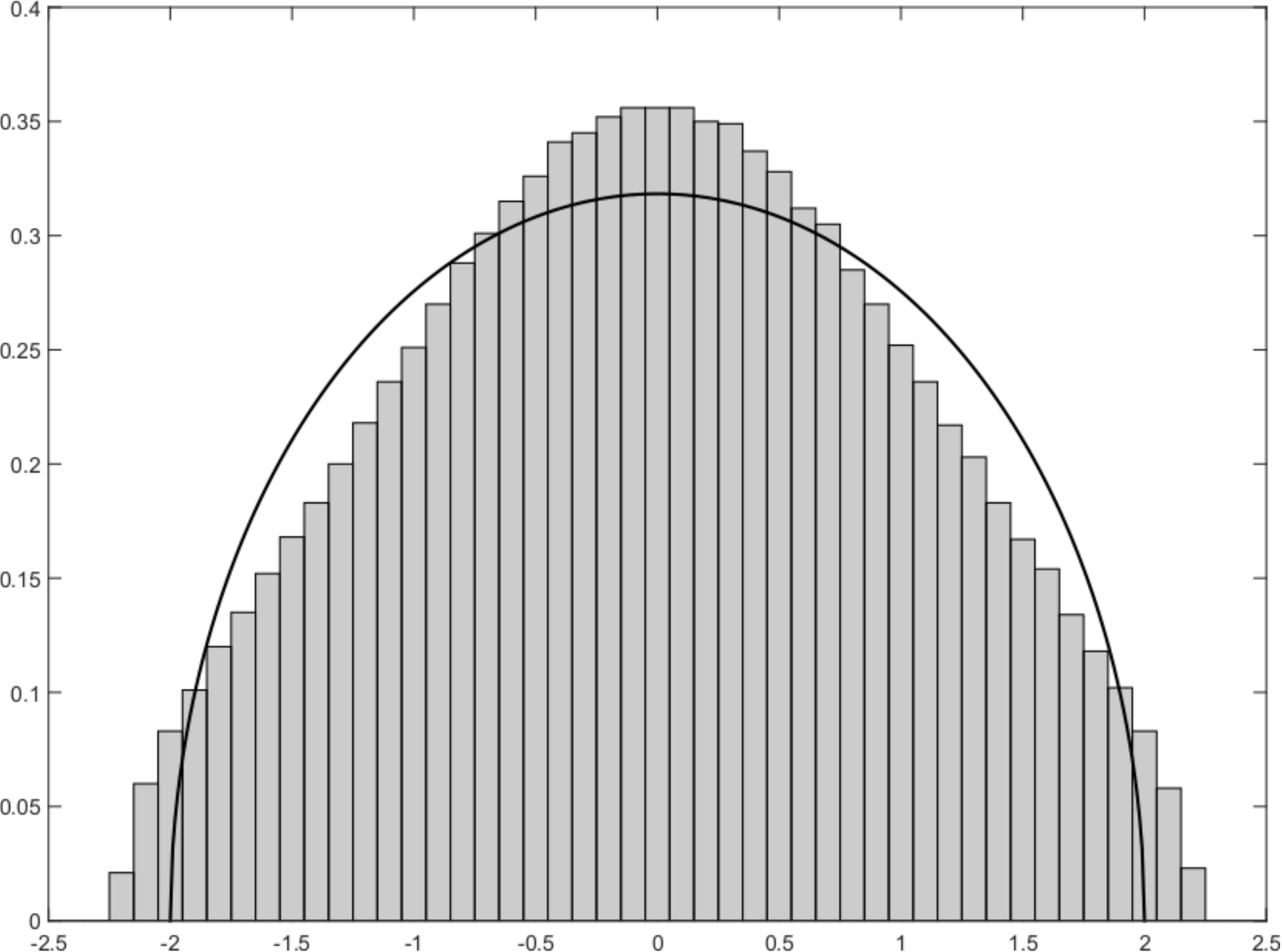}}
\hfill \text{}
\caption{Histograms of eigenvalues of a $10000 \times 10000$ matrix, filled with a Markov chain $(Z_n)_n$ with state space $\{-1,1\}$ and $\mathbb P (Z_n=i|Z_{n-1}=i)=0.7$. The line shows the density of the semi-circle.}
\label{different_fillings}
\end{figure}
\begin{lemma}
\label{lemma_expect}
Let $Z_n$ denote a Markov chain with state space $S=\{-1,+1\}$ and transition matrix $$\begin{pmatrix} p & 1-p \\ 1-p & p \end{pmatrix}$$
for some $p\in(0,1)$. Let $\mathbb P(Z_1=1)=\mathbb P(Z_1=-1)=\frac{1}{2}$. Further set  $\beta \coloneqq 2p-1$. Then we have for $k \in \mathbb N, i_1 \leq i_2 \leq \ldots \leq i_k$ and $n_j \coloneqq i_{j+1}-i_j$
\begin{equation*}
\mathbb E (Z_{i_1} Z_{i_2} \ldots Z_{i_k}) = \begin{cases} \beta^{n_1+n_3+\ldots + n_{k-1}}, & k \text{ even} \\ 0, & \text{else} \end{cases}.
\end{equation*}
\end{lemma}

\begin{proof}
The trick is similar to the previous proof. However, for a binary Markov chain we can compute explicitly that
$$
\E(Z_{i+j}|Z_i)=(2p-1)^j Z_i.
$$
Therefore, using that $s^2=1$ for all $s \in S$ we obtain for $i_1 < i_2 < \ldots <i_k$
\begin{eqnarray*}
\mathbb E (Z_{i_1} Z_{i_2} \cdots Z_{i_{k}})&=& \mathbb E (Z_{i_1} Z_{i_2} \cdots Z_{i_{k-1}} \mathbb E (Z_{i_{k}}|Z_{i_{k-1}}))\\
&=& (2p-1)^{i_k-i_{k-1}} \mathbb E (Z_{i_1} Z_{i_2} \cdots Z_{i_{k-2}}  Z^2_{i_{k-1}}) \\
&=& (2p-1)^{i_k-i_{k-1}} \mathbb E (Z_{i_1} Z_{i_2} \cdots Z_{i_{k-2}}).
\end{eqnarray*}
Repeating this calculation successively, we obtain
$$
\mathbb E (Z_{i_1} Z_{i_2} \ldots Z_{i_k}) = \begin{cases} \beta^{n_1+n_3+\ldots + n_{k-1}}, & k \text{ even} \\ 0, & \text{else} \end{cases}
$$
as asserted.
\end{proof}

\subsection{Gibbs measures}
The second important class of examples of random variables that satisfy Assumption \ref{bedingungen_proc} is provided by some Gibbs measures in one dimension.
We prove
\begin{lemma}\label{Gibbslemma}
Let $Z_n$ denote a stochastic process given by a one-dimensional Gibbs measure for some shift-invariant potential
$$
\Phi=\{\phi_A: A\subset \mathbb{Z}, 0<|A|<\infty\}
$$
on a finite state space (here we are following the Georgii's definition of a Gibbs measure in \cite{Georgii}).
Assume that for all odd numbers $k$ and all $i_1, \ldots, i_k$ we have that
$$
\mathbb E (Z_{i_1} Z_{i_2} \ldots Z_{i_k})=0
$$
and $\mathbb E(Z_i^2)=1$ for all $i \in \mathbb N$.
Moreover, assume that Dobrushin's condition is satisfied, i.e.
\begin{equation}
\sum_{A: 0\in A} (|A|-1) \sup_{\zeta,\eta} |\phi_A(\zeta)-\phi_A(\eta)| < 2.
\label{dobrushin}
\end{equation}
holds.

Finally we require the following condition
\begin{equation}
\sum_{A: 0\in A} e^{t \mathrm{diam}(A)} (|A|-1) \sup_{\zeta,\eta} |\phi_A(\zeta)-\phi_A(\eta)| < \infty,
\label{expcond}
\end{equation}
for some $t>0$.

Then for all even $k$ there are constants $C>0$ and $\beta \in[0,1)$ such that for all $i_1 \le i_2 \le \ldots \le i_k$ with  $n_j \coloneqq i_{j+1}-i_j$
 it holds
\begin{equation*}
\left|\mathbb E (Z_{i_1} Z_{i_2} \ldots Z_{i_k})\right|   \le C \beta^{n_1+n_3+\ldots + n_{k-1}}
\end{equation*}
and
for $j_1 \le \ldots \le j_k$ with $d \coloneqq \min_{n,m \in \{1,\ldots,k\}}|i_n-j_m|$  we have
 \begin{equation*}
\left| \mathbb E (Z_{i_1} \ldots Z_{i_{k}} Z_{j_1} \ldots Z_{j_{k}}) - \mathbb E (Z_{i_1} \ldots Z_{i_{k}}) \mathbb E (Z_{j_1} \ldots Z_{j_{k}})   \right| \leq C  \beta^{d}.
\end{equation*}
Moreover, for   $d'\coloneqq \min_{j=1,3,5\ldots} i_{j+2}-i_j$ it holds
\begin{equation*}
\left| \mathbb E (Z_{i_1}^2 Z_{i_3}^2 \ldots Z_{i_{k-1}}^2)-1 \right| \leq C  \beta^{d'}.
\end{equation*}
\end{lemma}

\begin{proof}
Note that the crucial estimate in the proof of Lemma \ref{MClemma1} was that for any $l \le k$  and all $i_1 \le i_2 \le \ldots \le i_k$  we have
\begin{eqnarray*}
&&\max_{s_{j_1}, \ldots, s_{j_k} \in S } |\P(Z_{i_k}=s_{j_k}, \ldots, Z_{i_l}=s_{j_l}|Z_{i_{l-1}}=s_{j_{l-1}}, \ldots, Z_{i_{1}}=s_{j_{1}} )\\
&&\qquad \qquad -\P(Z_{i_k}=s_{j_k}, \ldots, Z_{i_l}=s_{j_l}|\le C \alpha^{i_l-i_{l-1}}
\end{eqnarray*}
for some constant $C>0$ and some $\alpha \in [0,1)$.
Now due to \cite{Georgii}, Chapter 8, this estimate  holds if conditions \eqref{dobrushin} and \eqref{expcond} are satisfied.
Thus the rest of the proof follows the proof of Lemma \ref{MClemma1}.
\end{proof}

\begin{remark}
If we take, for example, a potential with finite range, then \eqref{expcond} is satisfied. If we consider a Gibbs measure with a parameter $\beta>0$ (usually called the inverse temperature), i.e.~we substitute $\Phi$ by $\Phi_\beta \coloneqq \{\beta\phi_A: A\subset \Z, 0<|A|<\infty\}$, then \eqref{dobrushin} holds whenever $\beta$ is small enough.
\end{remark}

\subsection{Gaussian processes}
Also Gaussian Markov processes satisfy the assumptions of Theorem \ref{main_theo}. More precisely we prove:
\begin{lemma}\label{gaussian}
Let $(Z_n)$ denote a stationary Gaussian Markov process with zero mean and variance one.
For  $k \in \mathbb N$ there are constants $C>0$ and $\beta \in[0,1)$ such that for all  $i_1 \le i_2 \le \ldots \le i_k$ with $n_j \coloneqq i_{j+1}-i_j$ and $j_1 \le \ldots \le j_k$ with $d \coloneqq \min_{n,m \in \{1,\ldots,k\}}|i_n-j_m|$   it holds
\begin{equation}
\label{even_moments_GP}
\left|\mathbb E (Z_{i_1} Z_{i_2} \ldots Z_{i_k})\right|   \le C \beta^{n_1+n_3+\ldots + n_{k-1}}, \quad \text{if } k \text{ even},
\end{equation}
\begin{equation}
\label{odd_moment_GP}
\left|\mathbb E (Z_{i_1} Z_{i_2} \ldots Z_{i_k})\right|  =0,\quad \text{if } k \text{ odd}
\end{equation}
and
 \begin{equation}
 \label{corr_GP}
\left| \mathbb E (Z_{i_1} \ldots Z_{i_{k}} Z_{j_1} \ldots Z_{j_{k}}) - \mathbb E (Z_{i_1} \ldots Z_{i_{k}}) \mathbb E (Z_{j_1} \ldots Z_{j_{k}})   \right| \leq C  \beta^{d}.
\end{equation}
Moreover,
for  $d'\coloneqq \min_{j=1,3,5\ldots} i_{j+2}-i_j$ we have
\begin{equation}
\label{second_moment_Gaussian_process}
\left| \mathbb E (Z_{i_1}^2 Z_{i_3}^2 \ldots Z_{i_{k-1}}^2)-1 \right| \leq C  \beta^{d'}.
\end{equation}

\end{lemma}

\begin{proof}
First, observe that assertion (\ref{odd_moment_GP}) is obvious.
Assertions \eqref{even_moments_GP}, \eqref{corr_GP} and \eqref{second_moment_Gaussian_process} follow from Isserli's theorem together with Doob's theorem. Indeed, in view of \eqref{even_moments_GP}, Isserlis' theorem \cite{Isserlis} states that
for all even $k$ and all indices $i_1 \le i_2 \le \ldots \le i_{2k}$
\begin{equation}\label{isser}
	\E\left[Z_{i_1}\cdot\ldots\cdot Z_{i_{2k}}\right] = \prod_{l=1}^{k} t(i_{2l-1} ,i_{2l}) + \sum_{\sigma\in S_{2k}^k} \prod_{l=1}^{k} t(i_{\sigma(2l-1)} ,i_{\sigma(2l)}),
\end{equation}
where we denote by $S_{2k}^k$ the set of all permutations $\sigma$ of $\{1,\ldots,2k\}$ such that $\{\sigma(2l-1), \sigma(2l)\}\neq \{2l'-1,2l'\}$ for at least one $l\in\{1,\ldots,k\}$ and any $l'\in\{1,\ldots,k\}$ and
for any $i,i'\in\N$,
\begin{equation*}
t(i,i') \coloneqq \cov(Z_i,Z_{i'}),
\end{equation*}
are the covariances.

As a consequence of Doob's theorem \cite{doob} we can conclude that for Gaussian Markov process
$t(i,i') = \beta^{|i-i'|}$ for some $-1 < \beta < 1$ and any $i,i'\in\N$. Thus the summands in the second term on the right hand side of \eqref{isser} are all smaller than the first summand. Together with the fact, that the size of $S_{2k}^k$ is finite and only depends on $k$ this shows \eqref{even_moments_GP}.
By similar arguments, we obtain \eqref{corr_GP} and \eqref{second_moment_Gaussian_process}.
\end{proof}

\subsection{A `diagonal' mapping $\varphi_N$}
In the following example and lemma, we show that there are fillings of a random matrix by a stochastic process that satisfy Assumption \ref{schranke_eintr_1}. Fortunately, the situation we always had in mind, namely the matrix we obtain when writing the entries of the stochastic process successively on the diagonals of the matrix as indicated by remark \ref{rem_filling} satisfies Assumption \ref{schranke_eintr_1}.

\begin{example}\label{example_diag}
We consider  the  mapping $\varphi_N$ that leads to the following matrices:
\begin{equation*}
X_N=\begin{pmatrix}
 Z_1 & Z_{N+1} & Z_{2N} &\ldots & \ldots & Z_{\frac{N(N+1)}{2}}
 \\
Z_{N+1} & Z_2 & Z_{N+2} &\ldots & \ldots & \ldots
  \\
Z_{2N} & Z_{N+2} & Z_3 & Z_{N+3} & \ldots & \ldots
   \\
\ldots & \ldots & Z_{N+3} & Z_4 & \ldots & \ldots
    \\
\ldots &\ldots & \ldots & Z_{10} & \ldots & Z_{2N-1}
     \\
Z_{\frac{N(N+1)}{2}} & \ldots & \ldots & \ldots & Z_{2N-1} & Z_N
\end{pmatrix}.
\end{equation*}
To check whether Assumption \ref{schranke_eintr_1} is satisfied, we need to compute the distance defined in Definition \ref{Def_filling}, (b).
The main ingredient to calculate this distance is a representation of the inverse $\varphi_N^{(-1)}$.
In this case, it can be calculated via
\begin{equation*}
\varphi_N^{(-1)}(j,k)=   \frac{N(N+1)}{2} - \frac{(N-|j-k|)(N-|j-k|+1)}{2} + \min(j,k).
\end{equation*}
For $j\leq k$ this formula can easily be verified as follows: The first term $\frac{N(N+1)}{2}$ denotes the number of steps between the $(1,1)$-entry and the $(1,N)$-entry. Observe that $N-(k-j)$ is the length of the diagonal containing the $(j,k)$-entry. Hence, the difference of the first two terms is the number of steps from the $(1,1)$-entry to the upper left corner of the diagonal with the $(j,k)$-entry. Further there are $j$ steps along this diagonal to the $(j,k)$ entry.
\end{example}
We will now verify that this filling is indeed valid in the sense of Assumption~\ref{schranke_eintr_1}.

\begin{lemma}\label {lemma_diag}
For $\varphi_N$ as in Example \ref{example_diag}, we have
for all  $N \in\mathbb N$, $i,j \in\{1,\ldots, N\}$ and $n \in \mathbb N, n>0$
\begin{equation*}
\# \{ x \in \{1,\ldots N\} : \|(i,x)- (x,j) \|_{\varphi_N} =n\} \leq 4.
\end{equation*}
In particular, Assumption \ref{schranke_eintr_1}  is satisfied and for stochastic processes that satisfy Assumption \ref{bedingungen_proc} the limiting spectral density is given by the semi-circle.
\end{lemma}
\begin{proof}
Let $N, i,j$ be fixed. We assume without loss of generality that $i<j$ (observe that $n\neq 0$ implies  $i\neq j$).
We consider
$$F(x)\coloneqq \|(i,x)- (x,j) \|_{\varphi_N}.$$
The main idea is that $\{1,\dots, N\}$ can be divided into four subsets, where $F$ is strictly monotone on each of the subsets. Hence, for given $n$,  $F(x)=n$ for at most four values of $x$. We consider these subsets individually  and recall the representation of $\varphi_N^{(-1)}$ given in example \ref{example_diag}, i.e.
\begin{align*}
F(x)
&=\left|   \frac{(N-|i-x|)(N-|i-x|+1)-(N-|x-j|)(N-|x-j|+1)}{2} \right. \\
& \qquad   +\min(i,x)- \min(x,j) \Big|
\\
&= \left| \frac{  i^2-j^2 + 2x(j-i)- (2N+1) \,  \left[ \, |i-x| - |j-x| \, \right] } {2} \right.\\
& \qquad +\min(i,x)- \min(x,j)\Big|.
\end{align*}

\underline{Case 1: $x \leq i$}
We will show that $F(x)_{x\leq i}$ is strictly increasing.
We have
\begin{equation*}
F(x)
=    (j-i) \, \left| -\frac{1}{2}(i+j) + x+N+\frac{1}{2} \right|
=  x(j-i) + C(i,j,N)
\end{equation*}
for some constant $C(i,j,N)$, which does not depend on $x$.
Hence, $F(x)$ is obviously strictly increasing for $x \leq i$.

\underline{Case 2: $x \geq j$}
We have
\begin{equation*}
F(x)
=(j-i)\left|        x- \frac{1}{2}(j+i) - N - \frac{3}{2}\right|
 = -x (j-i) + C(i,j,N).
\end{equation*}
Here, again  $C(i,j,N)$ denotes a constant, which does not depend on $x$ (though it may differ from case 1). Hence, $F(x)$ is strictly decreasing in  this case.

\underline{Case 3: $i < x < j$}
We have by an easy calculation
\begin{align}
&F(x) = \left| \left((j-i) - (2N+1)\right) \left( x-\frac{1}{2} (i+j) \right) + i-x \right| \nonumber
\end{align}
On the one hand, as $ (j-i) - (2N+1) \leq 0$ and $i-x \leq 0$, the term within the absolute value is negative for
$x \geq \frac{1}{2} (i+j)$.
On the other hand, the term is positive for $x<  \frac{1}{2} (i+j)$ (this follows from $(2N+1)-(j-i) \geq x-i$).
Hence we have for some constant $C(i,j,N)$
\begin{equation*}
F(x) = \begin{cases} \phantom{-}((j-i)-(2N+2))x - C(i,j,N),& x< \frac{1}{2} (i+j)\\ -((j-i)-(2N+2))x + C(i,j,N), & x \geq \frac{1}{2} (i+j)
\end{cases},
\end{equation*}
which shows that $F(x)$ is decreasing for $i<x<\frac{1}{2}(i+j)$ and increasing for $\frac{1}{2}(i+j) \leq x < j$.
\end{proof}

\section{Proof of the main result}
Before we complete the proof of the main result, we introduce some notation.
For $k \in \mathbb N$ we set
 \begin{equation*}
 S_k \coloneqq \{P=(P_1,\ldots, P_k) : P_i=(p_i,p_{i+1})\in \{1,\ldots, N\}^2 \}.
 \end{equation*}
 Then we can write (with $k+1=1$)
\begin{align}
\label{Darst_Ewert}
\mathbb E \left( \frac{1}{N} \text{tr} A_N^k \right) &= \frac{1}{N^{k/2 +1}} \sum_{P \in S_k} \mathbb E \left[ X_N(P_1) X_N(P_2) \ldots X_N(P_k) \right]
\\
&= \frac{1}{N^{k/2 +1}} \sum_{P \in S_k} \mathbb E \left[ X_N(P) \right], \nonumber
\end{align}
where we further abbreviated $X_N(P)\coloneqq  X_N(P_1) X_N(P_2) \ldots X_N(P_k)$.
In order to prove Theorem  \ref{main_theo} by the method of moments, it suffices to
prove the following lemma (see \cite[Lemma 2.1.6 and Lemma 2.1.7]{agz}).

\begin{lemma} \label{momentlemma}
For
$X_N \in \mathbb R^{N \times N}$  generated by a stochastic process  $Z_n$ and a mapping $\varphi_N$ as in  Theorem \ref{main_theo}  and  $A_N(i,j)\coloneqq \frac{1}{\sqrt{N}} X_N (i,j)$ we have
\begin{itemize}
\item[(i)]
\begin{equation*}
\lim_{N \to \infty}\mathbb E \left( \frac{1}{N} \tr A_N^k \right) =\begin{cases} 0& k \text{ odd} \\\kappa_{k/2}, & k \text{ even} \end{cases}.
\end{equation*}
\item[(ii)]
\begin{equation*}
\lim_{N \to \infty} \frac{1}{N^{k +2}} \sum_{P \in S_k, Q \in S_k} \mathbb E[X_N(P) X_N(Q)]- \mathbb E[X_N(P)] \mathbb E [X_N(Q)]=0
\end{equation*}
\end{itemize}
\end{lemma}

\begin{proof}
We start with the proof of (i) and consider the representation given in \eqref{Darst_Ewert}.
By (\ref{moment_cond1}), we can assume that $k$ is even, as the   term  in \eqref{Darst_Ewert} vanishes for $k$ odd.
For $P \in S_k$ we further introduce the notation
$$ [i]_P \coloneqq  \{ j\neq i :  P_i=P_j \text{ or } P_i=(p_{j+1},p_j) \}$$
and we write
\begin{itemize}
\item $ P \in S_k^0$, if $\#[i]_P=2$ for all $i=1,\ldots k$
\item $ P \in S_k^1$, if $\#[i]_P$ even for all $i=1,\ldots k$ and there exists $i_0$ such that $ \#[i_0]_P\geq 4$
\item  $ P \in S_k^*$, if there exists $i_0$ such that $\#[i_0]_P$ is odd.
\end{itemize}
We observe that by a simple combinatorial argument, we have
\begin{equation}
\label{card_S1}
\# \{P: P \in S_k^1\}=o(N^{k/2+1}).
\end{equation}
Indeed, assume that an equivalence relation on the indices $1, \ldots, k$ with equivalence classes $[i]$ is given such that each equivalence class has an even number of elements and   one equivalence class has at least four elements (the amount of such equivalence relations depends on $k$ only).
Now, we estimate the number of $P$, such that $[i]=[i]_P$ for all $i$.
There are at most $N^2$ choices for $(p_1,p_2)$. When determining the values $p_3,p_4,\ldots$ successively, there are at most $N$ choices for each $p_{j+1}$ where $j \notin [i]$ for $i<j$  (otherwise there are at most 2 possibilities). Observe, that there are at most $\frac{k}{2}-1$ equivalence classes in order to ensure $P \in S_k^1$.
Hence, we can neglect those terms with $P \in S_k^1$ in \eqref{Darst_Ewert} and we have
\begin{equation}
\label{Darst_Ewert2}
\mathbb E \left( \frac{1}{N} \text{tr} A_N^k \right) = \frac{1}{N^{k/2 +1}} \sum_{P \in S_k^0 \cup S_k^*} \mathbb E \left[ X_N(P_1) X_N(P_2) \ldots X_N(P_k) \right] + o(1).
\end{equation}
Next, we argue that the terms with $P \in S_k^0$ give a contribution of $\kappa_{k/2}$ in (\ref{Darst_Ewert2}), i.e.~we show
\begin{align}
\label{S_0-terms}
 \frac{1}{N^{k/2 +1}} \sum_{P \in S_k^0  } \mathbb E \left[ X_N(P_1) X_N(P_2) \ldots X_N(P_k) \right] =\kappa_{k/2} + o(1).
\end{align}
For $k=2$ the proof of (\ref{S_0-terms}) is immediate from
\begin{align*}
  \frac{1}{N^{2}} \sum_{P \in S_2^0 } \mathbb E \left[ X_N(P_1)^2\right]=1=\kappa_{1}.
\end{align*}
The main ingredient for the proof of (\ref{S_0-terms}) for $k\geq 4$ is \eqref{est_squares_2}, i.e.~for $i_1 <i_3< \ldots < i_{k-1}$ with $d\coloneqq \min_{j=1,3,5\ldots} i_{j+2}-i_j$ we have
\begin{equation}
\label{est_squares_2.1}
\mathbb E (Z_{i_1}^2 Z_{i_3}^2 \ldots Z_{i_{k-1}}^2)= 1+ \mathcal O( \beta^{d}).
\end{equation}
Here, the constant implicit in the $\mathcal O$-term does not depend on $i_1,i_3,\ldots i_{k-1}$.
In order to use (\ref{est_squares_2.1}), we need to introduce a sorting procedure that allows us to write $X_N(P_1)\ldots X_N(P_k)$ in terms of the stochastic process $(Z_n)_n$ in increasing order of the indices.
We set
 \begin{align*}
 &G_N: (\{1,\ldots, N\}^2)^k \to \left \{1,\ldots, \frac{N(N+1)}{2}\right \}^k;
 \\
 & G_N(P) \coloneqq \text{sort}(\varphi_N^{(-1)}(P_1), \varphi_N^{(-1)}(P_2), \ldots, \varphi_N^{(-1)}(P_k))
 \end{align*}
Here, $\text{sort}(\cdot): \mathbb R^k \to \mathbb R^k$ denotes the function that permutes the arguments such  that they are in increasing order.
Hence, for $P\in S_k$ with $G_N(P)=(i_1,\ldots,i_k),$ we have
\begin{equation}
\label{sort}
  X_N(P_1) X_N(P_2) \ldots X_N(P_k) =Z_{i_1} Z_{i_2} \ldots Z_{i_k}, \quad i_1\leq \ldots \leq i_k.
\end{equation}
In view of (\ref{est_squares_2.1}), we set for $P \in S_k^0$ with $G_N(P)=(i_1,i_1,\ldots,i_{k-1}, i_{k-1})$
 \begin{equation*}
 d(P) \coloneqq \min_{j=1,3,5\ldots} \, (i_{j+2}-i_j).
 \end{equation*}
 With this notation and (\ref{est_squares_2.1}) we obtain
 	\begin{equation*}
 	\frac{1}{N^{k/2 +1}} \sum_{P \in S_k^0  } \mathbb E \left[ 	
 	X_N(P)  \right] 
 	=  \frac{\# S_k^0 }	
 	{N^{k/2 +1}}+   \frac{1}{N^{k/2 +1}}\mathcal O\left( \sum_{P \in 	
 	S_k^0  }  \beta^{d(P)} \right),
	\end{equation*}
where the constant implicit in the $\mathcal O$-term is independent of $N$.
By the same combinatorial arguments like in the classical proof of Wigner's semicircle law by the moment method (see \cite{Arnold} or \cite{agz})
we have
\begin{equation*}
\lim_{N\to \infty}  \frac{\# S_k^0 }{N^{k/2 +1}} = \kappa_{k/2}.
\end{equation*}
Hence, to show (\ref{S_0-terms}), it remains to prove
 \begin{equation}
 \label{sum_d_P}
 \frac{1}{N^{k/2 +1}} \sum_{P \in S_k^0  }  \beta^{d(P)}=o(1).
 \end{equation}
We distinguish between $P \in  S_k^0$ with $d(P) \geq \sqrt{N}$ and with $d(P)< \sqrt{N}$. For $d(P) \geq \sqrt{N}$, we use $\beta^{d(P)} \leq \beta^{\sqrt{N}}$ and
by $\# S_k^0 = \mathcal O(N^{k/2 +1})$, we have
\begin{equation*}
 \frac{1}{N^{k/2 +1}} \sum_{P \in S_k^0, \, d(P)\geq \sqrt{N}} \beta^{d(P)}=o(1).
\end{equation*}
For the remaining terms with $d(P)\leq \sqrt{N}$ we estimate $\beta^{d(P)}\leq 1$ and hence for (\ref{sum_d_P}) it is sufficient to prove $\# \{P \in S_k^0 :  d(P) <\sqrt{N} \}=o(N^{k/2 +1}).$
We can proceed very similarly as in the proof of (\ref{card_S1}). Neglecting the additional information $d(P) <\sqrt{N}$ for a moment, we consider $k/2$ equivalence classes $[i]$ and determining $p_1,p_2,\ldots$ successively, there are $N$ choices for each `new' equivalence class and $N$ additionally  choices for the `starting point', say $p_1$. This would lead to an upper bound of $N^{k/2 +1}$, which is not sufficient. However, the restriction $d(P) <\sqrt{N}$  means that there are two equivalence classes $[i]$ and $[j]$, $i<j$ (again the number if such possible pairs of equivalence classes depends on $k$ only) such that
$\| P_i- P_j \|_{\varphi_N} \leq \sqrt{N}  $.
Hence, in our procedure, when we already determined $P_i$ and first encounter the equivalence class $[j]$, we actually have only $2\sqrt{N}$ choices rather than $N$. This reduces the bound for the total number of possibilities to $N^{(k+1)/2}$, which shows  
 $$\# \{P \in S_k^0 :  d(P) <\sqrt{N} \}=o(N^{k/2 +1})$$ 
and hence completes the proof of (\ref{sum_d_P}) resp.~of (\ref{S_0-terms}).
So far, we have
\begin{equation}
\label{reduction_to_S_k*}
\mathbb E \left( \frac{1}{N} \text{tr} A_N^k \right) =\kappa_{k/2} + \ \frac{1}{N^{k/2 +1}} \sum_{P \in   \cup S_k^*} \mathbb E \left[ X_N(P_1) X_N(P_2) \ldots X_N(P_k) \right] + o(1)
\end{equation}
and it remains to show that the second term on the r.h.s.~vanishes as $N \to \infty.$

When sorting $X_N(P_1)X_N(P_2)\ldots$ as in (\ref{sort}), we recall that by (\ref{Beh_EWert})  the differences $i_2-i_1, i_4-i_3, \ldots $ are of particular interest.
We model this structure with the help of pair partitions.
Note that, in the following paragraphs we will use partitions in a slightly different way then in the above paragraph.
Recall that in the derivation of \eqref{reduction_to_S_k*}, $i \sim j$ corresponded to  $(p_i,p_{i+1})=(p_j,p_{j+1})$ or $(p_i,p_{i+1})=(p_{j+1},p_j)$. From now on, we use $i\sim j$ to model that after sorting according to  (\ref{sort}), $X_N(P_i)$ and $X_N(P_j)$ are neighbors.
We introduce
\begin{equation*}
\mathcal{PP}(k)\coloneqq \{  \pi: \pi \text{ is a pair partition of  }\{1,\ldots, k\}\}.
\end{equation*}
For $\pi \in \mathcal{PP}(k)$, we say that $P \in S_k$ with $(i_1,\ldots,i_k)=G_N(P)$ (see (\ref{sort}))  is $\pi$-consistent , if
\begin{equation*}
 i\sim_\pi j \Leftrightarrow \{ \varphi_N^{(-1)}(P_i),\varphi_N^{(-1)}(P_j)\} = \{i_l,i_{l+1}\} \quad \text{for some odd }l  .
\end{equation*}
The $\pi$-consistency of  $P$, we write $P\in S_k(\pi)$, means that
 $\pi$ prescribes the pairs $P_i,P_j$ that correspond to an odd and the preceding even position of the vector obtained after applying $G_N$ to $P$.

We further introduce the notation for $n_1,\ldots, n_k \in \mathbb N$ and $\pi \in \mathcal{PP} (k)$
\begin{multline*}
 m(n_1,n_3,\ldots, n_{k-1}, \pi) \\
\coloneqq \# \{P \in S_k(\pi) :    (G_N(P))_{j+1}-(G_N(P))_{j}=n_j, j=1,3,5\ldots \}. 
\end{multline*}
Then $m$ is the number of possible vectors $P$, consistent with $\pi$, such that, after sorting $ X_N(P_1) X_N(P_2) \ldots X_N(P_k) $ in the order prescribed by the underlying stochastic process,   the first and the second terms are separated by $n_1$ steps of the stochastic process, the third and the fourth term are separated by $n_3$ steps in the stochastic process and so on.
We obtain by (\ref{Beh_EWert}) that there are constants $C>0$ and $\beta \in [0,1)$ such that
\begin{align*}
  &  \left|\frac{1}{N^{k/2 +1}} \sum_{P \in S_k^*;  } \mathbb E \left[ X_N(P_1) X_N(P_2) \ldots X_N(P_k) \right] \right|
  \\
 \le & C\frac{1}{N^{k/2 +1}} \sum_{\pi\in \mathcal{PP}(k)} \sum_{\substack{n_1, n_3, \ldots \geq 0, \\ \exists i: n_i \neq 0}} m(n_1,n_3\ldots, n_{k-1},\pi) \beta^{n_1+n_3+\ldots + n_{k-1}}.
 \end{align*}
 Recalling that $0 \leq \beta <1$, we have $ \sum_{n_1, n_3, \ldots \geq 0 }  \beta^{n_1+n_3+\ldots + n_{k-1}} \leq K$ for some constant $K$. Hence it suffices to show that for any $n_1, n_3, \ldots \geq 0$ with some $n_i>0$ and any pair partition $\pi$, we have
 \begin{equation}
\label{schranke_m}
 m(n_1,\ldots, n_{k-1},\pi)=o(N^{k/2 +1}).
\end{equation}
To prove (\ref{schranke_m}), let  $n_1, n_3, \ldots \geq 0$  be fixed and without loss of generality we assume that $n_1>0$.
We distinguish between crossing partitions $\pi$ and non-crossing partitions.
Here, a partition is said to be crossing if there are indices $i,i',j,j'$ with
\begin{equation}
\label{def_corssing}
 i \sim_\pi j, \quad  i' \sim_\pi j' \quad \text{and} \quad i<i'<j<j'.
\end{equation}
First, we show that for all $n_1,\ldots, n_{k-1}$ and all crossing pair partitions $\pi$ we have
\begin{equation}
\label{count_cross}
 m(n_1,\ldots, n_{k-1},\pi) \leq C N^{k/2}
\end{equation}
by estimating the number of possible choices for $P$ that contribute to the number $ m(n_1,\ldots, n_{k-1},\pi)$.
We note that the number of possibilities to sort the $\frac{k}{2}$ partition blocks depends on $k$ only, and we may hence assume that each partition block is associated to one of the $n_1,n_3,\ldots$.
Since $\pi$ is a crossing partitions there are indices $i,i',j,j'$ with (\ref{def_corssing}).
Without loss of generality we can assume that $i=1$.
An example for $k=14, i=1,j=9, i'=2, j'=14$ is shown in Fig.~\ref{fig_1}.
\begin{figure}
\begin{center}
 \begin{tikzpicture}[scale=.5]
\filldraw [gray]
(0,0) circle (3pt) node[anchor=north] {$p_1$}
(0,2) circle (3pt)node[anchor=east] {$p_2$}
(0,4) circle (3pt)node[anchor=east] {$p_3$}
(2,6) circle (3pt)node[anchor=south] {$p_4$}
(4,6) circle (3pt)node[anchor=south] {$p_5$}
(6,6) circle (3pt)node[anchor=south] {$p_6$}
(8,6) circle (3pt)node[anchor=west] {$p_7$}
(10,4) circle (3pt)node[anchor=west] {$p_8$}
(10,2) circle (3pt)node[anchor=west] {$p_9$}
(10,0) circle (3pt)node[anchor=north] {$p_{10}$}
(8,-2) circle (3pt)node[anchor=north] {$p_{11}$}
(6,-2) circle (3pt)node[anchor=north] {$p_{12}$}
(4,-2) circle (3pt)node[anchor=north] {$p_{13}$}
(2,-2) circle (3pt)node[anchor=north] {$p_{14}$}
;
 \draw (0,0) -- (0,2) -- (0,4) -- (2,6) -- (4,6) -- (6,6) -- (8,6) -- (10,4) -- (10,2) -- (10,0) -- (8,-2) -- (6,-2) -- (4,-2) -- (2,-2) -- (0,0);
 \draw[dotted] (0,1) -- (10,1) node[midway, above]{$n_3$};
  \draw[dotted] (0,3) .. controls (2,2) .. (1,-1) node[midway, above]{$n_5$};
   \draw[dotted] (3,-2) -- (5,6) node[midway, above]{$n_1$};
   \draw[dotted] (3,6) .. controls (2,4) .. (1,5) node[midway, above]{$n_7$};
      \draw[dotted] (7,6) .. controls (7,3) .. (9,-1) node[midway, above]{$n_9$};
            \draw[dotted] (9,5) .. controls (9,1) .. (7,-2) node[midway, above]{$n_{11}$};
                        \draw[dotted] (10,3) .. controls (6,1.6) .. (5,-2) node[midway, above]{$n_{13}$};
\end{tikzpicture}
\end{center}
\caption{Example of a crossing partition with $k=14$; dotted lines indicate the pair partition.}
\label{fig_1}
\end{figure}
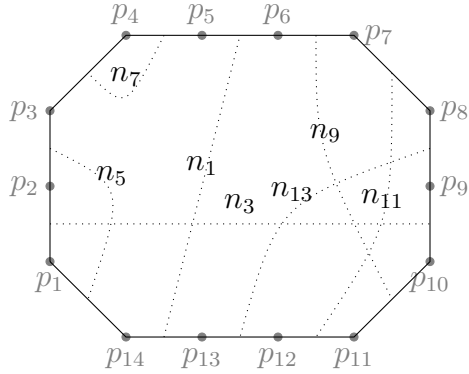

The main idea is to choose $p_1,p_2,\ldots$ successively (in an appropriate order), where once we fixed $P_i=(p_i,p_{i+1})$ there are only $4$ possible choices for $P_j$ if $i\sim j$ (as we already decided which value of $n_1,n_3,\ldots$ is associated to this pair, say $n_1$, we either have to choose $P_j$ such that it is $n_1$ steps before or $n_1$ steps after $P_i$ in the stochastic process, leading to 2 possibilities; considering the symmetry of the matrix we have a total of 4 possibilities for $P_j$ if $P_i$ is fixed).
We proceed in the following order: We chose $P_1=(p_1,p_2)$, for which there are at most $N^2$ possibilities.
We further choose $P_2,P_3,\ldots,P_{i-1}=(p_{i-1},p_i)$ with $N$ possibilities whenever we encounter a `new' equivalence class and 4 possibilities otherwise.  As $1\sim j$, we have only 4 possibilities for $P_j$. Now, we proceed further with $P_{j-1},P_{j-2},\ldots, P_{i+1}$ in the same way as before. Then we already determined $p_i$ and $p_{i+1}$ (and hence $P_i$). We continue in the same way with the remaining points. Hence, we had $N$ possibilities for $p_1$ and  $N$ possibilities for each of the $\frac{k}{2}$ equivalence classes, except for the equivalence class containing $i$. Hence, we had $CN^{\frac{k}{2}}$ possible choices, which proves (\ref{count_cross}).

It remains to show that
\begin{equation}
\label{count_non_cross}
 m(n_1,\ldots, n_{k-1},\pi)=o(N^{k/2+1}), \quad \text{ if } \pi \in \mathcal{PP}(k) \text{ is  non-crossing}.
\end{equation}
We recall that   $n_1 \neq 0$.  We observe that each non-crossing partition has a block of the form $\{i,i+1\}$, since for any $i \sim j, j>i+1 $ there exist $i<i'<j'<j$ with $i'\sim j'$. If the associated $n_i$ value is zero, we have $p_i=p_{j+1}$ and hence we can first determine the sequence $P_1,P_2,\ldots, P_{i-1}, P_{i+2}, \ldots$ and obtain the original sequence by further $N$ choices for $p_i$.  After `eliminating' $P_i, P_{i+1}$ we obtain again a non-crossing partition (we relabel the indices accordingly). We can repeat this elimination procedure, say for $r$ steps, until we arrive at (again after relabeling)
$ (Q_1, Q_2,\ldots , Q_{k-2r})$, with $Q_i=(q_i, q_{i+1})$  and a pair partition $\pi'$ on $\{1,\ldots, k-2r\},$
where all values of $n_i$ associated to partition blocks $\{i,i+1\}$ strictly larger than zero.
We can restore $P_1,\ldots, P_k$ from $(Q_1, Q_2,\ldots , Q_{k-r})$ by $N^{r}$ choices.
Let $\{i,i+1\}$ be a partition block of $\pi'$ associated to some $n_j>0$ (such a block exists as $n_1>0$).
As before we choose $q_i, q_{i-1},q_{i-2},\ldots, q_{i+2}$, for which there are $C N^{1+\frac{k}{2}-r-1}$ possibilities ($N$ possibilities for $q_i$, then $N$ possibilities for each of the $\frac{k}{2}-r$ equivalence classes except the equivalence class $\{i,i+1\}$). For $q_{i+1}$ there are by Assumption \ref{schranke_eintr_1} only $o(N)$ possibilities. Together with the $N^{r}$ possibilities to obtain $P_1,\ldots, P_k$ from $(Q_1, Q_2,\ldots , Q_{k-r})$, this proves the claim in \eqref{count_non_cross} and hence completes the proof of statement (i).

We recall, that  statement (ii) reads
\begin{equation*}
\lim_{N \to \infty} \frac{1}{N^{k +2}} \sum_{P \in S_k, Q \in S_k} \mathbb E[X_N(P) X_N(Q)]- \mathbb E[X_N(P)] \mathbb E [X_N(Q)]=0.
\end{equation*}
The main idea is to use \eqref{est_corr_combined}, which is a consequence of \eqref{Beh_EWert} and \eqref{est_corr}. Hence we introduce the following notation for $P,Q \in S_k$
\begin{equation*}
d(P,Q)\coloneqq \min_{1\leq i,j \leq k}\|P_i- Q_j \|_{\varphi_N}.
\end{equation*}
Further we write $[P,Q] \coloneqq (P_1,\ldots, P_k,Q_1,\ldots,Q_k)$ and for a pair partition $\pi \in \mathcal{PP}(2k)$, we say that $[P,Q] \in S_k \times S_k$ with $(l_1,\ldots,l_{2k})=G_N([P,Q])$ (see (\ref{sort}))  is $\pi$-consistent (we write $[P,Q]\in S_{2k}(\pi)$) , if
\begin{equation*}
 i\sim_\pi j \Leftrightarrow \{ \varphi_N^{(-1)}[P,Q]_i,\varphi_N^{(-1)}([P,Q]_j)\} = \{l_n,l_{n+1}\} \quad \text{for some odd }n.
\end{equation*}
Analogously to the definition of $m$ in the proof of (i), for $n_1,\ldots, n_{2k}, \delta  \in \mathbb N$ and $\pi \in \mathcal{PP} (2k)$, we set
\begin{multline*}
  M(n_1,n_3,\ldots, n_{2k-1}, \delta, \pi) 
 \coloneqq \\ \# \{[P,Q] \in S_{2k}(\pi) :    (G_N([P,Q]))_{j+1}-(G_N([P,Q]))_{j}=n_j, j=1,3,5\ldots
\\  \text{ and } d(P,Q)=\delta\}
\end{multline*}
By \eqref{est_corr_combined}
there are constants $C>0$ and $\beta \in [0,1)$ such that
\begin{align*}
  & \left|  \frac{1}{N^{k +2}} \sum_{P \in S_k, Q \in S_k} \mathbb E[X_N(P) X_N(Q)]- \mathbb E[X_N(P)] \mathbb E [X_N(Q)] \right|
  \\
 \le & C\frac{1}{N^{k +2}} \sum_{\pi\in \mathcal{PP}(2k)} \sum_{n_1, n_3, \ldots, \delta \geq 0} M(n_1,n_3\ldots, n_{2k-1},\delta, \pi) \beta^{n_1+n_3+\ldots + n_{2k-1}+ \delta}.
 \end{align*}
 Again, by the finiteness of the geometric series, it suffices to show that for any $n_1,n_3\ldots, n_{2k-1},\delta, \pi$
 \begin{equation}
 \label{lim_M}
  M(n_1,n_3\ldots, n_{2k-1},\delta, \pi)=o(N^{k +2}).
 \end{equation}
 For given $n_1,n_3\ldots, n_{2k-1},\delta, \pi$, we estimate the number of elements $[P,Q]$ that contribute to $M(n_1,n_3\ldots, n_{2k-1},\delta, \pi)$ for which $\delta=d(P,Q)=\|P_{i'}- Q_{j'} \|_{\varphi_N}$, that means the minimal distance is obtained between $P_{i'}$ and $Q_{j'}$ (the number of such  possible indices $i'$ and $j'$ depends on $k$ only and may hence be neglected for the consideration of $N \to \infty$).
When choosing $[P,Q]$ we proceed in the following order: We start with $P_{i'}$ and first fix all values $P_{i'+1}, \ldots, P_{k}, P_1, \ldots, P_{i'-1}$. Then we proceed with $Q_{j'}$ and the remaining values of $Q$, i.e.~$Q_{j'+1}, \ldots, Q_k,Q_1,\ldots, Q_{j'-1}$. As before, we have $N^2$ choices for the starting point $P_i$ and $N$ choices whenever we encounter a 'new' equivalence class, otherwise there is only a constant number of choices. There are $k$ equivalence classes, resp. $k-1$ 'new' equivalence classes, once $P_{i'}$ is fixed.
Here it is crucial, that when we fixed all values of $P$, we have only a constant number of choices for $Q_{j'}$ (by the restriction $\delta=\|P_{i'}- Q_{j'} \|_{\varphi_N})$, i.e. the $CN^2$ possibilities for $(P_{i'}, Q_{j'})$ provide a starting point for both $P$ and $Q$ rather than having $N$ further possibilities for $Q_{j'}$.
This shows \eqref{lim_M} and hence completes the proof of the main theorem.
 \end{proof}

It remains to prove Proposition  \ref{non_conv}, which states that under certain conditions on the mixed moments of the stochastic process, the fourth moment of the trace does not converge to the fourth moment of the semicircle law, i.e. we show that for some $C>0$ and $N $ large enough we have 
\begin{equation*}
\mathbb E \left( \frac{1}{N} \text{tr} A_N^4 \right) \geq \kappa_2 +C.
\end{equation*}
\begin{proof}[Proof of Proposition \ref{non_conv}]
We need to estimate
$$
\frac{1}{N^{3}} \sum_{P \in S_4} \mathbb E \left[ X_N(P_1) X_N(P_2) X_N(P_3) X_N(P_4) \right]
$$
from below.
As in the proof of Lemma \ref{momentlemma}, terms of the form $\mathbb E(X_N(P_1)^4)$  give a vanishing contribution to the sum as $N \to \infty$.
Recall that  by the assumption of the proposition we have
\begin{equation*}
\mathbb E (Z_{i_1} Z_{i_2} Z_{i_3}Z_{i_4}  ) = C \beta^{ n_1+n_3}
, \quad
 \mathbb E (Z_{i_1}^2 Z_{i_2}^2)=1
\end{equation*}
for some $0<\beta <1$.
Hence, as again $\frac{\# \{P: P \in S_4^0\}}{N^3} \to \kappa_2$ for $N\to\infty$,
we obtain for any $\varepsilon >0$ and $N$ sufficiently large
\begin{align}
  &\frac{1}{N^{3}} \sum_{P \in S_4} \mathbb E \left[ X_N(P_1) X_N(P_2) X_N(P_3) X_N(P_4) \right] - \kappa_2
   \nonumber
   \\ \geq& C\frac{1}{N^{3}} \sum_{\pi \in \mathcal{PP}(k)} \sum_{n_1, n_3 \geq 0, \exists i: n_i \neq 0} m(n_1, n_{3}, \pi) \beta^{n_1+n_3} - \varepsilon
   \nonumber \\
   \geq & C\frac{1}{N^{3}}   m(1, 1, \pi') \beta^{2}- \varepsilon
   \label{est_fourth_moment}
 \end{align}
for $\pi'=\{\{1,2\}, \{3,4\}\}$. Observe that the estimate in \eqref{est_fourth_moment} is justified as $\beta>0$.
We have to estimate (from below) the number of possible choices for $p_1,\ldots, p_4$ with $\| P_1-P_2\|_{\varphi_N}=  \| P_3-P_4\|_{\varphi_N}=1$ (see Fig.~\ref{fig_2}). We will show that this number is larger than $CN^3$ for some positive constant $C$. We can assume that $p_3-p_1=1$ (this further reduces the number of choices). Then $P_1,P_2$ resp.~$P_3,P_4$ are neighbors according to Definition \ref{Def_filling}.
\begin{figure}
 \begin{center}
 \begin{tikzpicture}[scale=.6]
\filldraw [gray]
(0,0) circle (3pt) node[anchor=east] {$p_1$}
(0,4) circle (3pt)node[anchor=east] {$p_2$}
(4,4) circle (3pt)node[anchor=west] {$p_3$}
(4,0) circle (3pt)node[anchor=west] {$p_4$}
;
 \draw (0,0) -- (0,4) -- (4,4) -- (4,0)--(0,0) ;
\draw[dotted] (0,2) .. controls (2,2) .. (2,4) node[midway, above]{$1$};
\draw[dotted] (4,2) .. controls (2,2) .. (2,0) node[midway, below]{$1$};
\end{tikzpicture}
\end{center}
\caption{The number of possible values for $p_1,\ldots, p_4 \in \{1, \ldots, N\}$ such that   $\| P_1-P_2\|_{\varphi_N}=  \| P_3-P_4\|_{\varphi_N}=1$ (indicated by the dotted lines) is larger than $CN^3$ for some  $C>0$.}
\label{fig_2}
\end{figure}
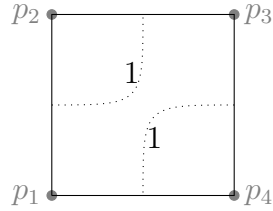
We set
\begin{multline*}
x_i \coloneqq \# \{j: (\varphi_N(j))_2=i, (\varphi_N(j+1))_2=i+1 \text{ or }\\
\shoveright{(\varphi_N(j))_2=i+1, (\varphi_N(j+1))_2=i\}\phantom{.}}\\
\shoveleft{y_i \coloneqq \# \{j: (\varphi_N(j))_1=i, (\varphi_N(j+1))_1=i+1 \text{ or }}\\
(\varphi_N(j))_1=i+1, (\varphi_N(j+1))_1=i\}.
\end{multline*}
Here, $x_i$ is the number of horizontal neighbors in the $i$-th and $i+1$-st column (in the upper triangular matrix) with distance one (according to $\|\cdot \|_{\varphi_N}$). Analogously, $y_i$ is the number of vertical neighbors  in the $i$-th and $i+1$-st row (in the upper triangular matrix) with distance one (according to $\|\cdot \|_{\varphi_N}$).
For each choice of $p_1,p_3$, say $p_1=i, p_3=i+1$, we have $(x_i+y_i)^2$ possibilities for $(p_2,p_4)$.
We observe that by the assumption of the lemma we have
$$ \sum_{i=1}^{N-1} x_i+y_i \geq C N^2.$$
By the Cauchy-Schwarz inequality we have
$$m(1,1, \pi') \geq\sum_i (x_i+y_i)^2 \geq \frac{1}{N} \left( \sum_i  x_i+y_i\right)^2 \geq C N^3.$$
Inserting this estimate into (\ref{est_fourth_moment}) proves the claim.
\end{proof}


\begin{thebibliography}{10}
\providecommand{\url}[1]{{#1}}
\providecommand{\urlprefix}{URL }
\expandafter\ifx\csname urlstyle\endcsname\relax
  \providecommand{\doi}[1]{DOI~\discretionary{}{}{}#1}\else
  \providecommand{\doi}{DOI~\discretionary{}{}{}\begingroup
  \urlstyle{rm}\Url}\fi

\bibitem{agz}
Anderson, G.W., Guionnet, A., Zeitouni, O.: An Introduction to Random Matrices.
\newblock Cambridge studies in advanced mathematics 118. Cambridge University
  Press, Cambridge (2010)

\bibitem{Arnold}
Arnold, L.: On {W}igner's semicircle law for the eigenvalues of random
  matrices.
\newblock Z. Wahrscheinlichkeitstheorie und Verw. Gebiete \textbf{19}, 191--198
  (1971)

\bibitem{brycdembo}
Bryc, W., Dembo, A., Jiang, T.: Spectral measure of large random {H}ankel,
  {M}arkov and {T}oeplitz matrices.
\newblock Ann. Probab. \textbf{34}(1), 1--38 (2006).
\newblock \doi{10.1214/009117905000000495}.
\newblock \urlprefix\url{http://dx.doi.org/10.1214/009117905000000495}

\bibitem{doob}
Doob, J.L.: The {B}rownian movement and stochastic equations.
\newblock Ann. of Math. (2) \textbf{43}, 351--369 (1942)

\bibitem{erdoes_survey}
Erd{\H{o}}s, L.: Universality of {W}igner random matrices: a survey of recent
  results.
\newblock Uspekhi Mat. Nauk \textbf{66}(3(399)), 67--198 (2011).
\newblock \doi{10.1070/RM2011v066n03ABEH004749}.
\newblock \urlprefix\url{http://dx.doi.org/10.1070/RM2011v066n03ABEH004749}

\bibitem{ESY}
Erd{\H{o}}s, L., Schlein, B., Yau, H.T.: Local semicircle law and complete
  delocalization for {W}igner random matrices.
\newblock Comm. Math. Phys. \textbf{287}(2), 641--655 (2009).
\newblock \doi{10.1007/s00220-008-0636-9}.
\newblock \urlprefix\url{http://dx.doi.org/10.1007/s00220-008-0636-9}

\bibitem{FL_CW2}
Friesen, O., L{\"o}we, M.: On the limiting spectral density of symmetric random
  matrices with correlated entries.
\newblock In: Random matrices and iterated random functions, \emph{Springer
  Proc. Math. Stat.}, vol.~53, pp. 3--29. Springer, Heidelberg (2013).
\newblock \doi{10.1007/978-3-642-38806-4_1}.
\newblock \urlprefix\url{http://dx.doi.org/10.1007/978-3-642-38806-4_1}

\bibitem{FL_CW}
Friesen, O., L{\"o}we, M.: A phase transition for the limiting spectral density
  of random matrices.
\newblock Electron. J. Probab. \textbf{18}, no. 17, 17 (2013).
\newblock \doi{10.1214/EJP.v18-2118}.
\newblock \urlprefix\url{http://dx.doi.org/10.1214/EJP.v18-2118}

\bibitem{Diag}
Friesen, O., L{\"o}we, M.: The semicircle law for matrices with independent
  diagonals.
\newblock J. Theoret. Probab. \textbf{26}(4), 1084--1096 (2013).
\newblock \doi{10.1007/s10959-011-0383-2}.
\newblock \urlprefix\url{http://dx.doi.org/10.1007/s10959-011-0383-2}

\bibitem{Georgii}
Georgii, H.O.: Gibbs Measures and Phase Transitions.
\newblock De Gruyter Studies in Mathematics 9. Walter de Gruyter, Berlin (1988)

\bibitem{gotze_tikho}
G{\"o}tze, F., Tikhomirov, A.N.: Limit theorems for spectra of random matrices
  with martingale structure.
\newblock Teor. Veroyatn. Primen. \textbf{51}(1), 171--192 (2006).
\newblock \doi{10.1137/S0040585X97982268}.
\newblock \urlprefix\url{http://dx.doi.org/10.1137/S0040585X97982268}

\bibitem{alice_stflour}
Guionnet, A.: Large random matrices: lectures on macroscopic asymptotics,
  \emph{Lecture Notes in Mathematics}, vol. 1957.
\newblock Springer-Verlag, Berlin (2009).
\newblock \doi{10.1007/978-3-540-69897-5}.
\newblock \urlprefix\url{http://dx.doi.org/10.1007/978-3-540-69897-5}.
\newblock Lectures from the 36th Probability Summer School held in Saint-Flour,
  2006

\bibitem{Kirsch_et_al}
Hochst\"attler, W., Kirsch, W., Warzel, S.: Semicircle law for a matrix
  ensemble with dependent entries.
\newblock preprint, to appear in Journal of Theoretical Probability  (2014)

\bibitem{Isserlis}
Isserlis, L.: On a formula for the product-moment coefficient of any order of a
  normal frequency distribution in any number of variables.
\newblock Biometrika \textbf{12}(1/2), 134--139 (1918).
\newblock \urlprefix\url{http://www.jstor.org/stable/2331932}

\bibitem{Peres}
Levin, D.A., Peres, Y., Wilmer, E.L.: {Markov chains and mixing times}.
\newblock American Mathematical Society (2006)

\bibitem{Schenker_Schulz-Baldes}
Schenker, J., Schulz-Baldes, H.: Semicircle law and freeness for random
  matrices with symmetries or correlations.
\newblock Math. Res. Lett. \textbf{12}, 531--542 (2005)

\bibitem{Schubert_Hankel}
Schubert, K.: Spectral density for random matrices with independent
  skew-diagonals.
\newblock arXiv:1510.06448, Preprint, submitted  (2015)

\bibitem{taovu}
Tao, T., Vu, V.: Random matrices: universality of {ESD}s and the circular law.
\newblock Ann. Probab. \textbf{38}(5), 2023--2065 (2010).
\newblock \doi{10.1214/10-AOP534}.
\newblock \urlprefix\url{http://dx.doi.org/10.1214/10-AOP534}.
\newblock With an appendix by Manjunath Krishnapur

\bibitem{taovu2}
Tao, T., Vu, V.: Random matrices: Universality of local eigenvalue statistics.
\newblock Acta Mathematica \textbf{206}, 127--204 (2011).
\newblock \urlprefix\url{http://dx.doi.org/10.1007/s11511-011-0061-3}

\bibitem{Wigner_ursprung}
Wigner, E.P.: Characteristic vectors of bordered matrices with infinite
  dimensions.
\newblock Ann. of Math. (2) \textbf{62}, 548--564 (1955)

\bibitem{Wigner}
Wigner, E.P.: On the distribution of the roots of certain symmetric matrices.
\newblock Ann. of Math. \textbf{67}, 325--328 (1958)

\end{thebibliography}

\end{document}